\documentclass[12 pt]{amsart}
\usepackage{amscd,amsfonts,amssymb,amsmath}
\usepackage{tikz}
\usepackage{tikz-cd}
\usepackage{xcolor}
\usepackage{placeins}
\usepackage{hyperref}
\usepackage[cp1251]{inputenc}
\usepackage[english]{babel}
\usetikzlibrary{babel}
\newtheorem{theorem}{Theorem}[section]
\theoremstyle{definition}

\newtheorem{corollary}[theorem]{Corollary}

\newtheorem{proposition}[theorem]{Proposition}
\newtheorem{question}[theorem]{Question}

\newtheorem{definition}[theorem]{Definition}
\newtheorem{example}[theorem]{Example}
\newtheorem{remark}[theorem]{Remark}

\newcommand{\Aut}{\operatorname{Aut}}

\newcommand{\GL}{\operatorname{GL}}

\newcommand{\id}{\mathrm{id}}

\title{EXTENSIONS OF  BRAID GROUP REPRESENTATIONS TO  THE MONOID OF SINGULAR BRAIDS
}

\author{Valeriy G. Bardakov, Nafaa Chbili, Tatyana A. Kozlovskaya}

\begin{document}
\maketitle
\begin{abstract}
Given a representation  $\varphi \colon B_n \to G_n$ of the braid group $B_n$,  $n \geq 2$ into a  group $G_n$, we are  considering the  problem of whether it is possible to extend this representation to a representation  $\Phi \colon SM_n \to A_n$, where $SM_n$ is  the singular braid monoid and   $A_n$ is an associative algebra, in which the group of units contains $G_n$. We also  investigate the   possibility of extending    the representation $\Phi \colon SM_n \to A_n$ to a representation $\widetilde{\Phi} \colon SB_n \to A_n$ of the singular braid group $SB_n$.
On the other hand, given two  linear representations $\varphi_1, \varphi_2 \colon H \to GL_m(\Bbbk)$ of a group $H$ into a general linear group over a field $\Bbbk$, we  define   the defect of one of these representations with  respect to the other. Furthermore, we  construct a  linear representation of $SB_n$ which is an extension of the Lawrence--Krammer--Bigelow representation (LKBR) and compute the defect of this extension  with respect to the  exterior product of two extensions of the Burau representation.
Finally, we discuss how to derive   an invariant of  classical links from
 the Lawrence--Krammer--Bigelow representation.

 \textit{Keywords}: Braid group, monoid of singular braids,  group of singular braids, representations,
 Artin representation, linear representations, Burau representation,
Lawrence--Krammer--Bigelow representation.

 \textit{Mathematics Subject Classification 2020}: 20E07, 20F36, 57K12.
\end{abstract}
\maketitle
\tableofcontents

\section{Introduction}

The monoid of singular braids or the Baez-Birman monoid, $SM_n$, $n \geq 2$,  was introduced independently by J. Baez in \cite{Baez} and J. Birman in  \cite{Bir}.
This monoid $SM_n$  is generated by the standard   generators $\sigma_1^{\pm 1},$ $\sigma_2^{\pm 1}$, $\ldots, \sigma_{n-1}^{\pm 1} $ of the braid group $B_n$ in addition to  the singular generators $\tau_1$, $\tau_2$, $\ldots, \tau_{n-1}$ depicted in Figure \ref{figure}.
 It is shown in \cite{FKR} that  the  monoid $SM_n$ embeds into a group $SB_n$ that is said to be  the singular braid group.
The reader is referred  to  \cite{DG, DG1, G, V3} for more on the singular braid monoid and the singular braid group.

It is well known that the Artin representation of $B_n$ may be used to calculate the
fundamental group of knot complements while the Burau representation can be used to calculate
the Alexander polynomial of knots. In \cite{G}, Gemein   studied extensions of the Artin representation
and the Burau representation to the singular braid monoid and the relation between them which is
induced by Fox free calculus.

In \cite{DG1} Dasbach and Gemein investigated extensions of the Artin representation $B_n \to \Aut(F_n)$ and the Burau representation $B_n \to GL_n(\mathbb{Z}[t, t^{-1}])$ to $SM_n$ and found connections between these representations.  They also  showed  that a certain linear representation of   $SM_3$ is faithful.

Just as with braids and classical links,  closing a singular braid yields a  singular link. Thus, the extensions of the Artin representation and the Burau representation give rise to invariants of singular knots. Gemein \cite{G} studied  invariants coming from the extended Artin
representation. Indeed, he obtained an infinite family of group invariants, all of them in relation with the
fundamental group of the knot complement.

Recall that a group $G$ is said to be \textit{linear} if there exists a faithful representation of $G$ into the general linear group ${\rm GL}_m(\Bbbk)$ for some integer $m\geq 2$ and a field $\Bbbk$. 
In \cite{Ver01}, linear representations of  the virtual braid groups $VB_n$, and the welded braid groups $WB_n$ into $\mbox{GL}_n(\mathbb{Z}[t, t^{-1}])$ were constructed. These representations extend  the  Burau representation.

The Lawrence-Krammer-Bigelow representation  is one of the most famous linear representations of the braid group $B_n$.  Lawrence  \cite{Law}, constructed a family of representations of $B_n$. It was shown in \cite{Kra, Big} that one of these representations
is faithful for all $n \in \mathbb{N}$. This leads to a  natural question regarding the  linearity of the singular braid group $SB_n$. It is worth mentioning here that  a  linear representation of   $SM_3$ which  is faithful was constructed in \cite{DG1}. This representation is  an extension of the Burau representation.

It is a natural approach to construct an extension of the Lawrence-Krammer-Bigelow representation to $SB_n$. In the present article we discuss the construction of  such  extension. Notice   that in \cite{BarP}, the first author    constructed a linear representation $\rho \colon VB_{n}\mapsto GL(V_{m})$, of the virtual braid group $VB_n$, where $V_m$ is a free module of dimension $m=n(n -1)/2$ with a basis $\left\{ v_{i,j}\right\}$, $1\leqslant i<j\leq n$. This representation is not  an extension of the  Lawrence-Krammer-Bigelow representation of $B_n$.

In his pioneering work \cite{Jo}, V.F.R. Jones constructed the HOMFLY polynomial $P(q, z)$, an isotopy invariant of classical knots and links, using the Iwahori--Hecke algebras $H_n(q)$, the Ocneanu trace  and the natural surjection of the classical braid groups $B_n$ onto the algebras $H_n(q)$. In \cite{JuLa2}
the Yokonuma--Hecke algebras have been used for constructing framed knot and link invariants
following the method of Jones.

The relation between singular knots and singular braids is just the same as in the classical
case. A lot of papers are dedicated to the construction of invariants of singular links. For instance,
the HOMFLY and Kauffman polynomials were extended to 3-variable polynomials of
singular links by Kauffman and Vogel \cite{KV}. The extended HOMFLY polynomial was recovered by the construction of traces on singular Hecke algebras \cite{PR}.
Juyumaya and Lambropoulou  \cite{JL} used a similar approach to define invariants of singular links.

 A generalization of the Alexander polynomial for oriented singular links and pseudo-links was introduced in \cite{NOS}.
The Alexander polynomials of a cube of resolutions (in Vassiliev's sense) of a singular
knot were categorified in \cite{A}. Moreover, a 1-variable extension of the Alexander polynomial for singular links was categorified in \cite{OSS}. The generalized
cube of resolutions (containing Vassiliev resolutions as well as those smoothings at double points which preserve the orientation) was categorified in \cite{OS}. On the other hand, Fiedler \cite{Fd} extended the Kauffman state models of the Jones and Alexander polynomials to the context of singular knots.

 A singular link can be regarded as an embedding in $\mathbb{R}^3$ of a four-valent graph with rigid vertices. We can think of such vertices as being
rigid disks with four strands connected to it which turn as a whole when we flip the vertex by 180 degrees. It is well-known that  polynomial invariants of  classical links extend (in various ways) to invariants of rigid-vertex isotopy of  knotted four-valent graphs.

In \cite{CCC} a homomorphism of $SM_n$ into  the Temperley--Lieb algebra was  constructed leading to  a polynomial invariant of singular links which is an
 extended Kauffman bracket.
Also, in \cite{CCC} it  was shown how to define this  invariant,  by interpreting
singular link diagrams as abstract tensor diagrams and employing a solution to the
Yang--Baxter equation. For  classical links, this was  done by Kauffman  in \cite{Ka1}.

The theory of singular braids is related  to the theory of pseudo-braids. In particular,
it  was proved in \cite{BJW}  that the monoid of pseudo-braids is isomorphic to the singular braid monoid.
Hence, the group of the singular braids is isomorphic to the group of pseudo-braids. On the other side, the theory of pseudo-links is a quotient of the theory of singular  links by the singular first Reidemeister move.

The paper is organized as follows. In Section \ref{Int}, we recall some basic  definitions and  facts on braid group, singular braid monoid, and  Artin and Burau representations.
In Section \ref{LKBR}, we shall discuss the extension of the LKBR to the singular braid monoid. Extensions of other braid group representations are discussed in Section  \ref{BGExt}.
In Section \ref{Defect}, we shall study the defect of the extension of the LKBR  with
respect to  the   exterior product of two extensions of the Burau representations. Finally, some open questions and directions for further research are given in Section \ref{OP}.

{\bf Notations.} In this paper, we shall use the following notations and conventions. If $\varphi_*$ is a
 representation of the braid group, where $*$ is some index, such as  $A$, $B$, $LKB$, etc.,  corresponding to   Artin, Burau, Lawrence-Krammer-Bigelow, and so forth, then $\Phi_*$ denotes an extension of this representation to the singular braid monoid $SM_n$. Here, extension means that $\Phi_* |_{B_n} = \varphi_*(B_n)$. If all $\Phi_*(\tau_i)$ are invertible, then we obtain   a representation of the singular braid group $SB_n$ that we shall denote by~$\widetilde{\Phi}_*$.

{\bf Acknowledgments}.

V.G. Bardakov and T.A. Kozlovskaya are supported by the Russian Science Foundation (RSF 24-21-00102) for work in sections 3 and 5. N. Chbili is  supported by United Arab Emirates University, UPAR grant No. G00004167 for work in sections 4 and 6.

\bigskip

\section{Basic definitions} \label{Int}

In this section we recall some basic definitions and results needed in the sequel. More details can be found in  \cite{Artin, Bir1, Mar}.

The braid group $B_n$, $n\geq 2$, on $n$ strands can be defined as
the  group generated by $\sigma_1,\sigma_2,\ldots,\sigma_{n-1}$ with the defining relations
\begin{equation}
\sigma_i \, \sigma_{i+1} \, \sigma_i = \sigma_{i+1} \, \sigma_i \, \sigma_{i+1},~~~ i=1,2,\ldots,n-2, \label{eq1}
\end{equation}
\begin{equation}
\sigma_i \, \sigma_j = \sigma_j \, \sigma_i,~~~|i-j|\geq 2. \label{eq2}
\end{equation}
The geometric interpretation of  $\sigma_i$, its inverse $\sigma_{i}^{-1}$ and the unit $e$ of $B_n$ are depicted  in  Figure~\ref{figure1}.
\begin{figure}[h]
\includegraphics[totalheight=8cm]{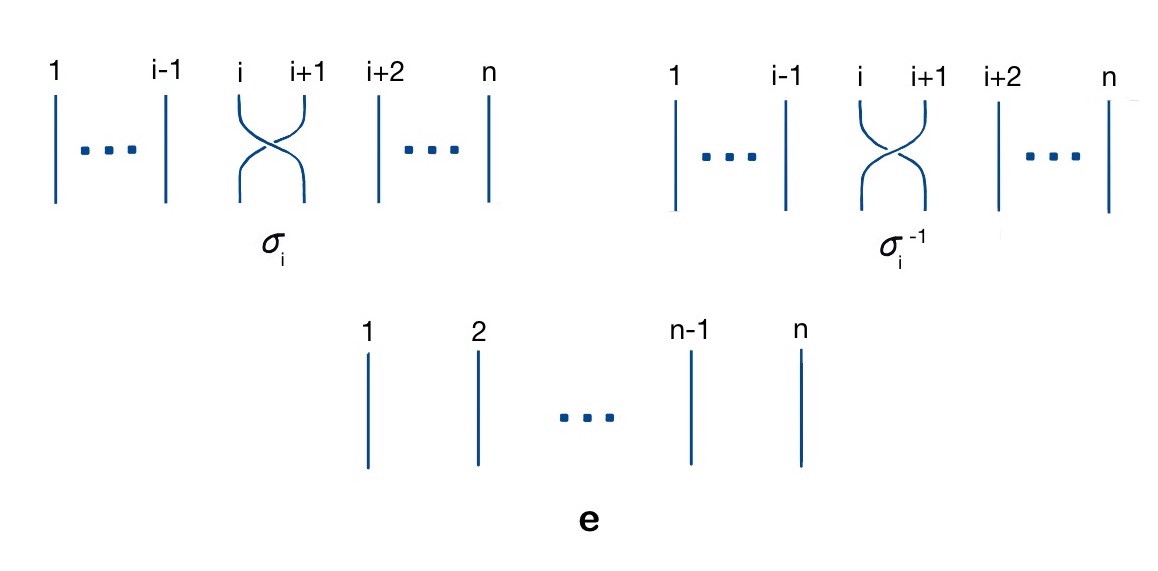}
\caption{The elementary braids $\sigma_i$, $\sigma_i^{-1}$ and the unit $e$.} \label{figure1}
\end{figure}

The group $B_n$ has a faithful representation into the automorphism group  ${\rm
Aut}(F_n)$ of the free group $F_n = \langle x_1, x_2, \ldots, x_n \rangle.$
In this case, the generator $\sigma_i$, $i=1,2,\ldots,n-1$, is mapped  to the automorphism
$$
\sigma_{i} \mapsto \left\{
\begin{array}{ll}
x_{i} \longmapsto x_{i} \, x_{i+1} \, x_i^{-1}, &  \\ x_{i+1} \longmapsto
x_{i}, & \\ x_{l} \longmapsto x_{l}, &  l\neq i,i+1.
\end{array} \right.
$$
This representation is known as the  Artin representation and is denoted hereafter  by $\varphi_A$.

Now, we shall define the {\it Burau representation}
$$
\varphi_B \colon B_n \longrightarrow GL(W_n)
$$
 of $B_n$, where $W_n$ is a free $\mathbb{Z}[t^{\pm 1}]$-module of rank $n$ with the basis $w_1, w_2, \ldots, w_n$.
The automorphisms $\varphi_B (\sigma_i)$, $i = 1, 2, \ldots, n-1$, of module $W_n$ act by the rule
$$
\varphi_B(\sigma_i) =
\left\{
\begin{array}{l}
w_i \longmapsto (1-t) w_i + t w_{i+1}, \\
w_{i+1} \longmapsto w_i, \\
w_k \longmapsto w_k,~~k \not=i, i+1.
\end{array}
\right.
$$

{\it The Baez--Birman monoid}
\cite{Baez, Bir} or {\it the singular braid monoid} $SM_n$ is generated
(as a monoid) by the elements $\sigma_i,$ $\sigma_i^{-1}$, $\tau_i$, $i = 1, 2, \ldots, n-1$.
The elements $\sigma_i,$ $\sigma_i^{-1}$  generate the braid group
$B_n$. The generators $\tau_i$  satisfy the defining relations
\begin{equation}
\tau_i \, \tau_j = \tau_j \, \tau_i, ~~~|i - j| \geq 2, \label{eq12}
\end{equation}
and the mixed  relations:
\begin{equation}
\tau_{i} \, \sigma_{j} = \sigma_{j} \, \tau_{i}, ~~~|i - j| \geq 2, \label{eq13}
\end{equation}
\begin{equation}
\tau_{i}  \, \sigma_{i} = \sigma_{i} \, \tau_{i},~~~ i=1,2,\ldots,n-1,  \label{eq14}
\end{equation}
\begin{equation}
\sigma_{i} \, \sigma_{i+1} \, \tau_i = \tau_{i+1} \, \sigma_{i}  \, \sigma_{i+1},~~~ i=1,2,\ldots,n-2,
 \label{eq15}
 \end{equation}
 \begin{equation}
\sigma_{i+1}  \, \sigma_{i} \, \tau_{i+1} = \tau_{i} \,
\sigma_{i+1} \, \sigma_{i}, ~~~ i=1,2,\ldots,n-2.
 \label{eq16}
\end{equation}

For a geometric interpretation of the elementary singular braid $\tau_i$ see Figure~\ref{figure}.

\begin{figure}[h]
\includegraphics[width=7cm, height=3.6cm]{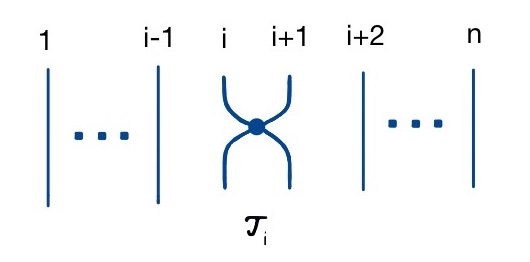}\\

\caption{The elementary singular braid  $\tau_i$.} \label{figure}

\end{figure}

It is proved by R. Fenn, E. Keyman and C. Rourke \cite{FKR} that the Baez-Birman monoid $SM_n$  is embedded into a group $SB_n$  which they call the  {\it singular braid group}.

\bigskip


\bigskip


\section{Extension of the Lawrence-Krammer-Bigelow representation}\label{LKBR}

The primary goal of this section is to find  extensions of the Lawrence-Krammer-Bigelow representation of the braid group $B_n$ to a representation of the singular braid monoid $SM_n$. In particular, we will  explicitely determine all such extensions in the cases $n=3$ and $n=4$.

Now, let us recall the definition of  the
 Lawrence-Krammer-Bigelow representation (LKBR for short) of the braid group $B_n$, see   \cite{Law, Kra, Big}. Let  $R =\mathbb{Z}[t^{\pm1}, q^{\pm1}]$ be the ring of Laurent polynomials on two variables $q$ and $t$ over the ring of integers. Let $V_m$ be a free module over $R$ with basis $\{v_{ij}\},$ $1 \leq i < j \leq n$. Then the   LKBR $\varphi_{LKB} \colon B_n \longrightarrow  GL(V_{m})$ is defined by action of $\sigma_i$, $i=1,2, \ldots, n-1$, on the basis of  $V_{m}$ as follows

\begin{equation}\label{LKBnnnn} \varphi_{LKB} (\sigma_i)(v_{k,l})= \begin{cases}
v_{k,l}, & \{k,l\} \cap \{i,i+1\} = \emptyset,\\
v_{i,l}, & k = i +1,\\
tq(q-1)v_{i,i+1}+(1-q)v_{i,l}+qv_{i+1,l} , & k = i \, \mbox{and} \, i+1< l,\\
tq^{2}v_{i,i+1},  &  k = i \, \mbox{and} \,  l=  i+1,\\
v_{k,i} ,  &  l = i+1 \, \mbox{and} \,  k<i,\\
(1-q)v_{k,i}+qv_{k,i+1}+q(q-1)v_{i,i+1}, &  l=i.\
\end{cases}
\end{equation}

As usual, we can present linear transformations $\varphi_{LKB} (\sigma_i)$ by matrices of size  $m \times m$ in the basis $v_{ij}$, $1 \leq i < j \leq n$. Notice that we are considering coordinates of vectors as  rows and the basis vectors of $V_m$ as  columns. We have  an isomorphism $\GL(V_n) \cong \GL_m(R)$, hence we  can consider LKBR as a homomorphism $\varphi_{LKB} \colon B_n \to \GL_m(R)$.

\begin{example} \label{Ex}
1) Under the representation $\varphi_{LKB} \colon B_3\longrightarrow GL_3(\mathbb{C})$  the  generators of $B_3$ are mapped  to the matrices,

$$ \sigma_{1} \mapsto \begin{pmatrix}
    tq^2	&  0	&0 \\
tq(q-1)	&1-q	&q \\
        0	&  1 &	0
\end{pmatrix},~~~ \sigma_{2} \mapsto \begin{pmatrix}
1-q	&q	&q(q-1) \\
  1	&0	   &   0 \\
  0	&0&	  tq^2
\end{pmatrix}. $$

2) Under the representation $\varphi_{LKB} \colon B_4 \longrightarrow GL_6(\mathbb{C})$ the  generators of $B_4$ are mapped  to the matrices,

$$ \sigma_{1} \mapsto \begin{pmatrix}
    tq^2 & 0 & 0 & 0 & 0 & 0  \\
    tq(q-1)& 1-q &  0& q & 0&0 \\
tq(q-1)& 0& 1-q & 0 &	q & 0\\
      0 & 1 & 0 & 0 & 0 & 0  \\
       0 & 0 & 1 & 0 &	0 & 0 \\
        0 & 0 & 0 & 0 &	0& 1
\end{pmatrix}, $$

$$ \sigma_{2} \mapsto \begin{pmatrix}
1-q	& q & 	0& 	  q(q-1)& 	  0& 	0 \\
  1	& 0	& 0	&         0	&   0& 	0 \\
  0	& 0	& 1	  &       0	&   0 & 	0 \\
  0	& 0 & 	0& 	    tq^2 & 	  0 & 	0 \\
  0 & 	0 & 	0 & 	tq(q-1) & 	1-q	& q \\
  0	& 0 & 	0	 &         0	&   1 & 	0
\end{pmatrix}, $$

$$ \sigma_{3} \mapsto \begin{pmatrix}
1	&  0 &	0	&  0	 &0	  &    0 \\
0	&1-q &	q&	  0	&0	&q(q-1) \\
0	&  1 &	0	&  0	&0	   &   0 \\
0	&  0	&0	&1-q&	q&	q(q-1) \\
0	&  0 &	0	&  1	&0	&      0 \\
0	&  0	&0	&  0	&0	&  tq^2
\end{pmatrix}. $$
\end{example}

To formulate our main result of this section, we will assume that the ring $R =\mathbb{Z}[t^{\pm1}, q^{\pm1}]$ is a subring of the complex numbers $\mathbb{C}$, where $t$ and $q$ are transcendental numbers over $\mathbb{Q}$ and $V_m$ is a vector space   over $\mathbb{C}$.

\begin{theorem} \label{TB}
Let $\varphi_{LKB} \colon B_n \longrightarrow  GL(V_{m})$ be the Lawrence-Krammer-Bigelow representation and  $u, v \in \mathbb{C}$.
Then the  map
$$
\Phi^{u,v}_{LKB} \colon SM_n \longrightarrow  GL(V_{m}),
$$
 which is defined on the generators  by the formulas
$$
\Phi^{u,v}_{LKB}(\sigma_{i}) =  \varphi_{LKB} \left( \sigma_{i} \right),
 $$
$$
\Phi^{u,v}_{LKB}(\tau_{i}) = u\varphi_{LKB} \left( \sigma_{i} \right)  +v e, ~~e = \id,
$$
defines a representation of $SM_n$ which  is an extension of the  LKBR of $B_n$. If all $\Phi^{u,v}_{LKB}(\tau_{i})$ are invertible, then we get a representation of the group $SB_n$. Moreover, for $n=3, 4$ any extension of the LKBR to $SM_n$ has this form.
\end{theorem}

\begin{proof}
It can be easily checked that the  transformations $\Phi^{u,v}_{LKB}(\sigma_{i})$ and $\Phi^{u,v}_{LKB}(\tau_{i})$, $i = 1, 2, \ldots, n-1$ satisfy all defining relations of $SM_n$. Hence, $\Phi^{u,v}_{LKB}$ defines
a representation of $SM_n$. Obviously, if all transformations $\Phi^{u,v}_{LKB}(\sigma_{i})$ are invertible, then we get a linear representation of the singular braid group $SB_n$. Now, it remains to prove that in the cases $n=3,4$ any
 extension of the LKBR to $SM_n$ is of the form $\Phi^{u,v}_{LKB}$.

Let us consider the case $n=3$.  We shall proceed as follows. Take as images of $\tau_1$ and $\tau_2$ two matrices of size  $3 \times 3$ with 9 unknown entries. Then,   include these matrices with the images of $\sigma_1$ and  $\sigma_2$ under the LKBR (see  Example \ref{Ex}(1)), into the defining relations of $SM_3$.  Elementary but tedious calculations  show that the  images of $\tau_1$ and $\tau_2$ must be the following

$$ \tau_{1} \mapsto \begin{pmatrix}
  uq^2t+v	&        0	&  0 \\
utq(q-1)	 &u(1-q)+v &	uq \\
          0	 &       u	  &v \\
\end{pmatrix},~~~
\tau_{2} \mapsto \begin{pmatrix}
u(1-q)+v	&        uq	&  uq(q-1) \\
u	 &v &	0 \\
          0	 &       0	  &uq^{2}t+v\\
\end{pmatrix}. $$

In the case $n=4$,  using the same calculations as for the case $n=3$ and  the matrices from Example \ref{Ex}(2), we should be able to prove that

$$ \tau_{1} \mapsto \begin{pmatrix}
    utq^2+v & 0 & 0 & 0 & 0 & 0  \\
    utq(q-1)& u(1-q)+v &  0& uq & 0&0 \\
utq(q-1)& 0& u(1-q)+v & 0 &	uq & 0\\
      0 & u & 0 & v & 0 & 0  \\
       0 & 0 & u & 0 &	v & 0 \\
        0 & 0 & 0 & 0 &	0& u+v
\end{pmatrix}, $$

$$ \tau_{2} \mapsto \begin{pmatrix}
u(1-q)+v	& uq & 	0& 	  uq(q-1)& 	  0& 	0 \\
  u	& v	& 0	&         0	&   0& 	0 \\
  0	& 0	& u+v	  &       0	&   0 & 	0 \\
  0	& 0 & 	0& 	    utq^2+v & 	  0 & 	0 \\
  0 & 	0 & 	0 & 	utq(q-1) & 	u(1-q)+v	& uq \\
  0	& 0 & 	0	 &         0	&   u & 	v
\end{pmatrix}, $$

$$ \tau_{3} \mapsto \begin{pmatrix}
u+v	&  0 &	0	&  0	 &0	  &    0 \\
0	&u(1-q)+v &	uq&	  0	&0	&uq(q-1) \\
0	&  u &	v	&  0	&0	   &   0 \\
0	&  0	&0	&u(1-q)+v&	uq&	uq(q-1) \\
0	&  0 &	0	&  u	&v	&      0 \\
0	&  0	&0	&  0	&0	&  utq^2+v
\end{pmatrix}. $$

\end{proof}

\begin{remark}
One may ask whether it is possible to find conditions under which $\det(\Phi^{u,v}_{LKB}(\tau_{i})) \not= 0$. Indeed, using the relations
$$
\tau_{i+1} =  \sigma_{i}   \sigma_{i+1}  \tau_{i}  \sigma_{i+1}^{-1}  \sigma_{i}^{-1}, ~~i = 1, 2, \ldots, n-2,
$$
we see that in $SM_n$ all $\tau_i$ are conjugate with $\tau_1$. Hence,
$$
\det(\Phi^{u,v}_{LKB}(\tau_{1})) =\det(\Phi^{u,v}_{LKB}(\tau_{2})) = \ldots = \det(\Phi^{u,v}_{LKB}(\tau_{n-1})).
$$
It means that it is enough to find only $\det(\Phi^{u,v}_{LKB}(\tau_{1}))$ in $B_3$, $B_4$ and so on.

In $B_3$ we have
$$
\det(\Phi^{u,v}_{LKB}(\tau_{1})) = (u q^2 t + v) (v^2 + v u (1 - q) - u^2 q).
$$

In $B_4$ we have
$$\det(\Phi^{u,v}_{LKB}(\tau_{i}))=4tu^6+v^6-2quv^5+q^2tuv^5+3uv^5+q^2u^2v^4-6qu^2v^4-$$
$$-2q^3tu^2v^4+3q^2tu^2v^4+3u^2v^4+3q^2u^3v^3-6qu^3v^3+$$
$$+q^4tu^3v^3-6q^3tu^3v^3+3q^2tu^3v^3 +u^3v^3+3q^2u^4v^2-2qu^4v^2+$$
$$+3q^4tu^4v^2-6q^3tu^4v^2+q^2tu^4v^2+q^2u^5v+3q^4tu^5v-2q^3tu^5v.$$

\end{remark}

\begin{remark}
Theorem \ref{TB} implies  the existence  of extensions of the LKBR to the singular braid group $SB_n$. In contrast, it has been  proved in  \cite{BRep} that there are no extensions of the LKBR  to the virtual braid group $VB_n$ nor to  the welded braid group $WB_n$ for $n \geq 3$.
\end{remark}

\subsection{Burau representation} We shall now  show that  some analogous of Theorem \ref{TB} holds for the Burau representation. We will assume that the Burau representation is a~representation,
$$
\varphi_B \colon B_n \to \GL_n(\mathbb{Z}[t^{\pm 1}]) \leq \GL_n(\mathbb{C})
$$
 into the general linear group over the field $\mathbb{C}$. Here we  take as $t$ some transcendental over $\mathbb{Q}$.
It was proved in \cite{G} that any linear local homogeneous representation $\Phi_B \colon S{M_n} \to \GL_n(\mathbb{C})$ that is an extension of the Burau representation of $B_n$ can be defined  on the generators:
$$\Phi_B ({\sigma _i}) = \left( {\begin{array}{*{20}{c}}
{{E_{i - 1}}}&\vline& {\begin{array}{*{20}{c}}
{\;\;0}&{\;\;0}
\end{array}}&\vline& 0\\
\hline
\begin{array}{l}
0\\
0
\end{array}&\vline& {\begin{array}{*{20}{c}}
{1 - t}&t\\
1&0
\end{array}}&\vline& {\begin{array}{*{20}{c}}
0\\
0
\end{array}}\\
\hline
0&\vline& {\begin{array}{*{20}{c}}
{\;\;0}&{\;\;0}
\end{array}}&\vline& {{E_{n - i - 1}}}
\end{array}} \right),\;\;\;$$

\begin{center}
$$\Phi_B ({\tau _i}) = \left( {\begin{array}{*{20}{c}}
{{E_{i - 1}}}&\vline& {\begin{array}{*{20}{c}}
{\;\;\;0}&{\;\;\;\;\;\;\;\;0}
\end{array}}&\vline& 0\\
\hline
{\begin{array}{*{20}{c}}
0\\
0
\end{array}}&\vline& {\begin{array}{*{20}{c}}
{1 - t + at}&{t - at}\\
{1 - a}&a
\end{array}}&\vline& {\begin{array}{*{20}{c}}
0\\
0
\end{array}}\\
\hline
0&\vline& {\begin{array}{*{20}{c}}
{\;\;\;0}&{\;\;\;\;\;\;\;\;0}
\end{array}}&\vline& {{E_{n - i - 1}}}
\end{array}} \right),$$
\end{center}
where $a \in \mathbb{C}$. If $a \not= 1/2$, then we  get a representation of $SB_n$.

In \cite{KSS}, it  was proved that the representation  $\Phi_B \colon S{M_n} \to \GL_n(\mathbb{C})$ is
reducible.  Furthermore,  a  reduced representation $\Phi_B^r \colon S{M_n} \to \GL_{n-1}(\mathbb{C})$ was constructed and was  proved to be  irreducible.

A proof of the following proposition  is straightforward.

\begin{proposition}

The images of the generators $\sigma_i$ and $\tau_i$ in the representation $\Phi_B \colon {SM_n} \to \GL_n(\mathbb{C})$, are related by the formulas
$$
\Phi_B(\tau_i) =  (1-a) \varphi_B(\sigma_{i})+ a \cdot \id, ~~i = 1, 2, \ldots, n-1.
$$

\end{proposition}

\bigskip


\section{Extensions of the braid group representations}\label{BGExt}

Suppose that we have a representation $\varphi \colon B_n \to G_n$ of the braid group into a group $G_n$. In this section, we discuss whether it is  possible to extend this representation to a representation $\Phi \colon SM_n \to A_n$, where $A_n$  is an associative algebra such that $G_n$ lies in the group of units $A_n^*$.

\begin{proposition} \label{Rep}
Let $\varphi \colon B_n \to G_n$ be a representation of  the braid group $B_n$, $\Bbbk$ be a field and $a, b, c \in \Bbbk$.
Then a  map  $\Phi_{a,b,c} \colon SM_n \to  \Bbbk[G_n]$ which acts on the generators by the rule
$$
\Phi_{a,b,c}(\sigma_i^{\pm 1})=\varphi(\sigma_i^{\pm 1}), ~~~\Phi_{a,b,c}(\tau_i)=a \varphi (\sigma_i)+b \varphi (\sigma^{-1}_i)+ce, ~~i = 1, 2, \ldots, n-1,
$$
defines a  representation  of $SM_n$ into $\Bbbk[G_n]$. Here $e$ is the unit element of $G_n$.
\end{proposition}

\begin{proof}
We need to  verify  that the defining relations of $SM_n$ are mapped  to the defining relations of $\Bbbk[G_n]$. Since this is true for the  defining relations of $B_n$, we have   to check  the mixed relations and relations which involve  only the generators $\tau_i$ (see relations (\ref{eq12})--(\ref{eq16})). At first, let us consider the relation (\ref{eq12}),
\begin{equation*}
\tau_i \, \tau_j = \tau_j \, \tau_i, ~~~|i - j| \geq 2.
\end{equation*}
Acting by $\Phi_{a,b,c}$, we get the equality
$$
(a \varphi(\sigma_{i})+b \varphi(\sigma_i^{-1})+ce) (a \varphi(\sigma_{j})+b \varphi(\sigma_{j}^{-1})+ce)  = (a \varphi(\sigma_{j})+b \varphi(\sigma_{j}^{-1})+ce) (a \varphi(\sigma_{i})+b \varphi(\sigma_i^{-1})+ce).
$$
Since,
$$
\varphi(\sigma_{i}^{\pm}) \varphi(\sigma_{j}) = \varphi(\sigma_{j}) \varphi(\sigma_{i}^{\pm}),~~\varphi(\sigma_{i}^{\pm}) \varphi(\sigma_{j}^{-1}) =
\varphi(\sigma_{j}^{-1}) \varphi(\sigma_{i}^{\pm}),
$$
the needed  relation holds.  Relations (\ref{eq13})--(\ref{eq14}) can be checked in a  similar way.

Let us  check the long relation (\ref{eq15}) (the checking of the last relation (\ref{eq16}) is similar),
$$
\sigma_{i} \sigma_{i+1} \tau_i = \tau_{i+1} \sigma_{i} \sigma_{i+1}.
$$
Taking the images by $\Phi_{a,b,c}$   of  both sides, we get
$$
\varphi(\sigma_{i}) \varphi(\sigma_{i+1})(a \varphi(\sigma_{i})+b \varphi(\sigma_i^{-1})+ce) = (a \varphi(\sigma_{i+1})+b \varphi(\sigma_{i+1}^{-1})+ce) \varphi(\sigma_{i} \sigma_{i+1}),
$$
which is equivalent to the relation
$$
a \varphi(\sigma_{i}) \varphi(\sigma_{i+1}) \varphi(\sigma_{i})+b \varphi(\sigma_{i}) \varphi(\sigma_{i+1}) \varphi(\sigma_i^{-1})+c \varphi(\sigma_{i}) \varphi(\sigma_{i+1})e =
$$
$$
=a \varphi(\sigma_{i+1}) \varphi(\sigma_{i}) \varphi(\sigma_{i+1})+b \varphi(\sigma_{i+1}^{-1}) \varphi(\sigma_{i}) \varphi(\sigma_{i+1}) +c \varphi(\sigma_{i}) \varphi(\sigma_{i+1}) e.
$$
Taking  into consideration relations of $B_n$ and the fact that $\varphi$ is a representation,   we  can easily  see  that
$$\Phi_{a,b,c}(\sigma_{i} \sigma_{i+1} \tau_i) =\Phi_{a,b,c}( \tau_{i+1} \sigma_{i} \sigma_{i+1}).$$
\end{proof}

Let us give some  examples of  representations of this type.

{\it Birman representation.} Motivated by the study of  invariants of finite type (or Vassiliev invariants) of classical knots,  Birman \cite{Bir} introduced a  representation of $SM_n$ into the
 group algebra $\mathbb{C}[B_n]$ by the expression
$$
\sigma_{i}^{\pm 1} \mapsto \sigma_{i}^{\pm 1},~~ \tau_i \mapsto  \sigma_{i}-  \sigma_i^{-1}, ~~i = 1, 2, \ldots, n-1.
$$

It is easy to see that if we put in Proposition \ref{Rep}, $\varphi = \id$, $a = 1$, $b=-1$, $c=0$,
we get $\Phi_{1,-1,0}$ that is  the Birman representation.
Paris  \cite{P} proved that this representation  is faithful.

A natural question that arises here is the following:

\begin{question}
For what values of  $a, b, c \in \mathbb{C}$ the representation $\Phi_{a,b,c}$ is faithful?
\end{question}

Further, we can formulate  a question about the possibility of  extending  the   representation  $\Phi_{a,b,c}$ to  the singular braid group $SB_n$.  To construct a representation of $SB_n$, it is required  that the image of $\tau_i$ has  an inverse, for all $i \in \{ 1, 2, \ldots, n-1 \}$. Let

$$
B = \sigma_i(a \sigma_i + c) + b+e.
$$
Using the formula
$$
(e - A)^{-1} = e + A + A^2 + A^3 + \ldots,
$$

we get
$$
\Phi_{a,b,c}(\tau_i)^{-1} = (a \sigma_i + b \sigma_i^{-1} + c e)^{-1} = \sigma_i (e - B + B^2 - \ldots ).
$$

Hence, we obtain  a representation
$$
\tilde{\Phi}_{a,b,c} \colon SB_n \to \mathbb{C}[[B_n]].
$$

\begin{question}
For what values of $a, b, c \in \mathbb{C}$ the representation $\tilde{\Phi}_{a,b,c}$ is faithful?
\end{question}

\medskip



\medskip

\bigskip


\section{Comparing  LKBR and the exterior square of  Burau representation}\label{Defect}

Suppose that we have two representations
$$
\varphi, \psi \colon G \to \GL_l(\Bbbk)
$$
of a group $G$ into a general linear group over a field $\Bbbk$. In order to  compare  these two representations we  introduce the following definition.

\begin{definition}
The {\it additive defect} of an element $g\in G$ is the matrix  $d_g = \varphi(g) - \psi(g)$.
The  {\it multiplicative  defect} of an element $g \in G$ is  the matrix $k_g = \varphi(g)^{-1}  \psi(g)$.
\end{definition}

\subsection{Tensor product of two Burau representations}
Consider  the Burau representation
$$
\varphi_B \colon B_n \to \GL(W_n),
$$
where $W_n$ is a vector space over   $\mathbb{C}$ with a basis $w_1$, $w_2$, $\ldots$, $w_{n-1}$.
Let us take the  second exterior power $\wedge^{2}$ $W_n$ that is  the quotient of $W_n \otimes W_n$ by the subspace generated by the set $\left\{ w\otimes w\mid \  w \in W_n \right\}.$
The vector space $\wedge^{2}$ $W_n$ has a basis
$$
u_{ij} = e_i \wedge e_j,~~~1 \leq i < j \leq n.
$$
We will denote by  $\varphi_{DB} \colon B_n \to GL(\wedge^{2}W_n)$ the  homomorphism which is defined on the generators of $B_n$  by the rule
$$
\varphi_{DB}(\sigma_k)(u_{ij}) = \varphi_B(\sigma_k)(e_i) \wedge \varphi_B(\sigma_k)(e_j),~~~1 \leq i < j \leq n,
$$
where $\varphi_B$ is the Burau representation of $B_n$.

Using elementary  calculations, one can prove the following:

\begin{proposition}
The generators of $B_n$ act on $\wedge^{2}W_n$ by automorphisms,
 $$
\varphi_{DB} \left( \sigma_{i} \right) =\begin{cases}u_{k i}\mapsto (1-q)u_{ki}+qu_{k i+1},&k<i ;\\ u_{k i+1}\mapsto u_{k i},& k<i ;\\ u_{i i+1}\mapsto (1-q)u_{i i+1} ;&\\ u_{i l} \mapsto (1-q) u_{i l}+q u_{i+1 l},& i+1<l ;\\u_{i+1 l} \mapsto u_{i l}; &\\ u_{k l} \mapsto u_{k l},   &
\left\{ k,l \right\}\cap  \left\{ i+1,i \right\}=\emptyset,      \end{cases}
 $$
 for all $i=1, 2, \ldots,n-1$.
\end{proposition}

Notice  that the vector spaces  on which act the representations $\varphi_{LKB}$ and $\varphi_{DB}$ are isomorphic. We are interested  in investigating  the connection between these two  representations. We can reformulate the general definition of the defect  as follows.

\begin{definition}
The additive defect of an element $w \in B_n$ is an element
$$
d_w = \varphi_{DB}(w) - \varphi_{LKB}(w).
$$
The multiplicative  defect of an  element $w \in B_n$ is an element
$$
k_w = \varphi_{DB}(w)^{-1}  \varphi_{LKB}(w).
$$
\end{definition}

Let us find the defect of the generators $\sigma_i$. Denote  $g_i = \varphi_{LKB}(\sigma_i)$ and  $h_i = \varphi_{DB}(\sigma_i)$, then the additive defect of $\sigma_i$ is equal to
 $ d_i=h_i-g_i$, and the multiplicative defect is equal to
$k_i=g_i^{-1}h_i$.

\begin{proposition}
The following formulas hold

   $$
g^{-1}_{i}  :\begin{cases}u_{k i}\mapsto u_{k,i+1},&k<i ;\\ u_{k i+1}\mapsto \frac{1}{q} u_{k i}+\frac{q-1}{q} u_{k i+1}& k<i ;\\ u_{i i+1}\mapsto \frac{-1}{q-1} u_{i i+1} ;&\\ u_{i l} \mapsto u_{i+1 l},& i+1<l; \\u_{i+1 l} \mapsto \frac{1}{q} u_{i l}+\frac{q-1}{q} u_{i+1 l},   &  i+1<l;\\ u_{k l} \mapsto u_{k l}, &
\left\{ k,l \right\}\cap  \left\{ i+1,i \right\}=\emptyset.      \end{cases}
 $$

  $$
d_{i}  :\begin{cases}v_{k i}\mapsto q(q-1)v_{i,i+1},&k<i ;\\ v_{k i+1}\mapsto 0& k<i ;\\ v_{i i+1}\mapsto (tq^{2}+q-1)v_{i i+1} ;&\\ v_{i l} \mapsto tq(q-1) v_{i i+1};& \\v_{i+1 l} \mapsto 0; &\\ v_{k l} \mapsto v_{k l},   &
\left\{ k,l \right\}\cap  \left\{ i+1,i \right\}=\emptyset;     \end{cases}
 $$

    $$
k_i  :\begin{cases}w_{k i}\mapsto w_{k,i},&k<i ;\\ w_{k i+1}\mapsto w_{k i+1}+(q-1)w_{k+1 i+1}& k<i ;\\ w_{i i+1}\mapsto -\frac{tq^{2}}{q-1} w_{i i+1} ;&\\ w_{i l} \mapsto w_{i l},& i+1<l; \\w_{i+1 l} \mapsto t(q-1) w_{i i+1}+ w_{i+1 l},   &  i+1<l;\\ w_{k l} \mapsto w_{k l}, &
\left\{ k,l \right\}\cap  \left\{ i+1,i \right\}=\emptyset.      \end{cases}
 $$

\end{proposition}

\begin{proof}
The proof is straightforward using routine calculations.
\end{proof}

We shall now calculate the additive and multiplicative defects in the cases $n=3$ and $n=4$.
\begin{example}
In the case  $n=3$ we have
 $$
  g_{1} =\begin{pmatrix}
1-q	&  0	& 0 \\
  0	&1-q	 &q \\
  0	 & 1	& 0
\end{pmatrix}~~~
g_{1}^{-1} =\begin{pmatrix}
\frac{-1}{q-1}	&  0	&   0 \\
  0&  0	 &    1 \\
 0&\frac{1}{q}&\frac{q-1}{q}
\end{pmatrix}
~~~ h_{1} =\begin{pmatrix}
    tq^2&  0	& 0\\
qt(q-1) &1-q &q \\
        0 & 1 &	0
\end{pmatrix}.
$$
Hence, the multiplicative and additive defects are equal to

\begin{center}
$
k_1 =  g_{1}^{-1} h_{1}  =\begin{pmatrix}
 \frac{-tq^2}{q-1}&0&0\\
 0&	1&	0\\
  t(q-1)&0&1 \\
\end{pmatrix},
$
$d_1 = \begin{pmatrix}q^2t+q-1&0&0\\
qt(q-1)&0&0\\
0&0&0
\end{pmatrix}.
$
\end{center}

For the image of $\sigma_2$ we have
 $$
  g_{2} =\begin{pmatrix}

1-q &q&  0\\
  1&0	 & 0& \\
  0&	0&1-q

\end{pmatrix},~~
 g_{2}^{-1} =\begin{pmatrix}
  0&  1& 0 \\
\frac{1}{q}&\frac{q-1}{q}&  0 \\
  0&0&\frac{-1}{q-1}
\end{pmatrix},~~~ h_{2} =\begin{pmatrix}

1-q&q &q(q-1)\\
   1&	0& 0\\
   0&	0&tq^2

\end{pmatrix} $$

Hence,
\begin{center}
$k_2 =  g_{2}^{-1} h_{2}  =\begin{pmatrix}
1&0& 0 \\
0&1& q-1\\
0&0&\frac{-tq^2}{q-1}
\end{pmatrix},$
$
d_2 = \begin{pmatrix} 0&0&q(q-1)\\
0&0&0\\
0&0&q^2t+q-1.
\end{pmatrix}
$
\end{center}
\end{example}

\begin{example}
In the case  $n=4$ we have

 $$ g_{1} =\begin{pmatrix}
1-q& 0&0& 0&0	&0 \\
  0&1-q &0& q&0&0 \\
  0& 0&1-q& 0&q&0 \\
  0& 1&0&0&0&0 \\
  0& 0 &1&  0&	0&0 \\
  0&  0&0& 0&0	&1
\end{pmatrix},~~~
g_{1}^{-1} =\begin{pmatrix}
\frac{-1}{q-1}&0&0&0&0&0 \\
       0&0&0& 1& 0&0\\
       0&0&0&0&1&0\\
       0&\frac{1}{q}&0&\frac{q-1}{q} &0&0\\
       0&0&\frac{1}{q} &0&\frac{q-1}{q}&0\\
       0&0&0 &0& 0&1\\
\end{pmatrix}.
 $$

 $$ h_{1} =\begin{pmatrix}
    tq^2 & 0 & 0 & 0 & 0 & 0  \\
    tq(q-1)& 1-q &  0& q & 0&0 \\
tq(q-1)& 0& 1-q & 0 &	q & 0\\
      0 & 1 & 0 & 0 & 0 & 0  \\
       0 & 0 & 1 & 0 &	0 & 0 \\
        0 & 0 & 0 & 0 &	0& 1
\end{pmatrix}. $$

Hence, the multiplicative and additive defects are equal to

\begin{center}
 $ k_1=g_{1}^{-1} h_{1}  =\begin{pmatrix}

\frac{-tq^2}{q-1}	&0&0&0&0&0 \\
           0&1&0&0&0&0\\
           0&0&1&0&0&0\\
       t(q-1)&0&0	&1&0&0\\
       t(q-1)&0&0&0&1&0\\
           0&0&0&0&0&1
\end{pmatrix}, \,\,\,
$
$
d_1=\begin{pmatrix}
q^2t+q-1&0&0&0&0&0 \\
qt(q-1)&0&0&0&0&0 \\
qt(q-1)&0&0&	0&0&0 \\
0&0&0&0&0&0 \\
0&0&0&0&0&0 \\
0&0&0&0&0&0
\end{pmatrix}.
 $
\end{center}
Let us consider the image of $\sigma_2$. We have
 $$ g_{2} =\begin{pmatrix}

1-q&	q& 0&0&0&0\\
  1&0&0&	0&0&0\\
  0&0&1&0&0&0\\\
  0&0&0&1-q&0&0\\
  0&0&0&0&1-q&q\\
  0&0&0&	0&1&0

\end{pmatrix},~~~
 g_{2}^{-1} =\begin{pmatrix}
  0&1& 0&	0&0&0 \\
\frac{1}{q}&\frac{q-1}{q}&0&0&0&0\\
  0&0	&1&0&0&0\\
  0&0	&0&\frac{-1}{q-1}&0&0\\
  0&0	&0&0&0&1\\
  0&0	&0&0&\frac{1}{q}&\frac{q-1}{q}
\end{pmatrix} $$

 $$ h_{2} =\begin{pmatrix}
1-q& q & 0&q(q-1)&0&0 \\
  1&0&0&0&0&0 \\
  0&0	&1&0&0& 	0 \\
  0&0&0&tq^2 &0 & 0 \\
  0&0&0&tq(q-1)&1-q&q \\
  0&0&0&0&1&0
\end{pmatrix}. $$
Hence, the multiplicative and additive defects are equal to

 \begin{center}
 $ k_2= g_{2}^{-1} h_{2}  =\begin{pmatrix}
1&0&0&0&0&0 \\
0&1&0&q-1&0&0 \\
0&0&1&0&0&0 \\
0&0&0& \frac{-tq^2}{q-1} &0&0 \\
0&0&0&0&1&0 \\
0&0&0&t(q-1)&0&1
\end{pmatrix}, \,\,\, $
$ d_2=\begin{pmatrix}
0&0&0&q(q-1)&0&0 \\
0&0&0&0&0&0 \\
0&0&0&0&0&0 \\
0&0&0&q^2t+q-1&0&0 \\
0&0&0&qt(q-1)&0&0 \\
0&0&0&0&0&0
\end{pmatrix}. $
\end{center}

For the image of $\sigma_3$,

 $$ g_{3} =\begin{pmatrix}
1&0&0&0&0&0 \\
0&1-q&q&	0&0&0 \\
0&1&0&0&0&0 \\
0&0&0&1-q&q&0 \\
0&0&0&1&0&0 \\
0&0&0&0&0&1-q
\end{pmatrix},~~~
 g_{3}^{-1} =\begin{pmatrix}
1&0&0&0&0&0 \\
0&0&1&0&0&0\\
0&\frac{1}{q}&\frac{q-1}{q}&0&0&0\\
0&0&0&0&1&0 \\
0&0&0&\frac{1}{q}&\frac{q-1}{q}&0 \\
0&0&0&0&0&\frac{-1}{q-1}
\end{pmatrix},
 $$
 $$ h_{3} =\begin{pmatrix}
1& 0&0&0	&0&0 \\
0&1-q &q&0&0&q(q-1) \\
0&1&0&0&0&0 \\
0&0&0&1-q&q&	q(q-1) \\
0&0&0&1&0&0 \\
0&0&0&0&0&tq^2
\end{pmatrix}. $$
Hence, the multiplicative and additive defects are equal to

\begin{center}
 $ k_3 = g_{3}^{-1} h_{3}  =\begin{pmatrix}
1&0&0&0&0&0 \\
0&1&0&0&0&0 \\
0&0&1&0&0&q-1 \\
0&0&0&1&0&0 \\
0&0&0&0&1&q-1 \\
0&0&0&0&0&\frac{-tq^2}{q-1}
\end{pmatrix},
 $
$d_3 =\begin{pmatrix}
0&0&0&0&0&0 \\
0&0&0&0&0&q(q-1) \\
0&0&0&0&0&0 \\
0&0&0&0&0&q(q-1) \\
0&0&0&0&0&0 \\
0&0&0&0&0&q^2t+q-1
\end{pmatrix}.
 $
\end{center}

\end{example}

\begin{remark}
According to  \cite{BJW} the monoid of singular braids $SM_n$ is isomorphic to the monoid of pseudo braids $PM_n$ and the group of  singular braids $SB_n$ is isomorphic to the group of pseudo braids $PG_n$. Hence, all representations of $SM_n$ and $SB_n$ give representations of $PM_n$ and $PB_n$, respectively.
\end{remark}

\bigskip


\section{Open problems and directions for further research}\label{OP}

\subsection{From the Lawrence-Bigelow-Krammer representation to knot invariants}
Using the Burau representation of the braid groups one can define
the Alexander polynomial  which  is a knot invariant of classical knots. To the best of our knowledge, there are no knot  invariants defined
  from  the Lawrence-Bigelow-Krammer representation. We suggest the following construction of such invariants.

Let $B_{\infty} = \cup_{n=1}^{\infty} B_n$. For any $\beta \in B_{\infty}$,  define the polynomial
$$
f_{\beta} = f_{\beta}(q, t, \lambda) = det(\varphi_{LKB}(\beta) - \lambda \cdot id) \in \mathbb{Q}[q,t, \lambda],
$$
that is the characteristic polynomial which corresponds to the image of  $\beta$ by  the  Lawrence-Bigelow-Krammer representation $\varphi_{LKB}$. Let
$$
F = \{ f_{\beta} ~|~\beta \in B_{\infty} \} \subseteq \mathbb{Q}[q,t, \lambda]
$$
be the set of such characteristic polynomials. We define an equivalence relation on $F$ as follows:
$$
f_{\beta} \sim_M f_{\gamma} \Leftrightarrow ~\mbox{there is a sequence of Markov moves which transforms} ~\beta~ \mbox{into}~ \gamma.
$$
Using the Markov theorem one can prove the following:

\begin{proposition} \label{Inv}
The equivalence class $[f_{\beta}]$ under the equivalence relation $\sim_M$ is an  invariant of the knot $\hat{\beta}$ that is the closure of the braid $\beta$.
\end{proposition}

\begin{question}
Which  knots   it is   possible to distinguish using the  invariant $[f_{\beta}]$?
\end{question}

By properties  of characteristic polynomials, $f_{\beta}$  does not change under the first  Markov move, i.~e.   $f_{\beta} =f_{\alpha^{-1} \beta \alpha}$ for all $\alpha,\beta \in B_n$.
Let $L = \hat{\beta}$ be a link that is the closure of a braid $\beta$. Define the following set of polynomials
$$
F_L = \{ f_{\gamma}~|~\gamma \in B_{\infty} ~\mbox{can be constructed from}~ \beta ~\mbox{using Markov moves}\}.
$$
From Proposition \ref{Inv}, it  follows.

\begin{corollary}
The set of polynomials $F_L$  is an  invariant of the link $L = \hat{\beta}$.
\end{corollary}

It is interesting to investigate  whether   it is  possible to find all polynomials in $F_L$. In the following example, we give some calculations.

\begin{example}
1) (Trivial knot) Let $\beta = \sigma_1 \sigma_2 \in B_3$ be a 3-strand braid. It is easy to check that its closure $\hat{\beta}$  is the trivial knot $U$. Also, one can see that the closure  of any of the 3-strand braids
$$
\sigma_1^{-1} \sigma_2,~~\sigma_1 \sigma_2^{-1},~~\sigma_1^{-1} \sigma_2^{-1},
$$
 gives the trivial knot. The corresponding polynomials have the form,
$$
f_{\sigma_1 \sigma_2}=q^6t^2-w^3,
$$

$$
f_{\sigma_1^{-1} \sigma_2}=(q^2t-q^2tw^3-qw^2+q^4t^2w^2-q^3t^2w^2-q^3tw^2+2q^2tw^2-qtw^2+w^2+qw-q^4t^2w+q^3t^2w+
$$
$$
+q^3tw-2q^2tw+qtw-w)/(q^2t),
$$

$$
f_{\sigma_1 \sigma_2^{-1}}=(q^2t-q^2tw^3-qw^2+q^4t^2w^2-q^3t^2w^2-q^3tw^2+2q^2tw^2-qtw^2+w^2+qw-q^4t^2w+q^3t^2w+
$$
$$
+q^3tw-2q^2tw+qtw-w)/(q^2t),
$$

$$
f_{\sigma_1^{-1} \sigma_2^{-1}}=(-q^6t^2w^3+1)/(q^6t^2).
$$

Also, the closure of the 4-strand braid $\sigma_1 \sigma_2  \sigma_3$ gives the trivial knot. For this braid,
$$
f_{\sigma_1 \sigma_2  \sigma_3}=q^{12}t^3+w^6-q^4tw^4-q^8t^2w^2.
$$

2) (Hopf link) Let $\beta = \sigma^2_1 \sigma_2 \in B_3$ be a 3-strand braid. It is easy to check that its closure $\hat{\beta}$  is the Hopf link $H$. We have
$$
f_{\sigma^2_1 \sigma_2}=-q^9t^3-w^3+q^3tw^2+q^6t^2w.
$$

3) (Trefoil knot) Let $\beta = \sigma^3_1 \sigma_2 \in B_3$ be a 3-strand braid. It is easy to check that its closure $\hat{\beta}$  is the trefoil knot  $T$. We have
$$
f_{\sigma^3_1 \sigma_2}=q^{12}t^4-w^3.
$$

\end{example}

\subsection{Extensions of the Artin representations}

In \cite{G} a family of extensions of the Artin representation of $B_n$ to the  monoid of the singular braids $SM_n$ is constructed.

\begin{question}
Is it possible to construct non-trivial extensions of the Artin representation of $B_n$ to the group of the singular braids $SB_n$? Is it possible to construct a faithful such representation?
\end{question}

\subsection{Representation into the Temperley--Lieb algebra.}
For each integer $n \geq 2$, the $n$-strand Temperley--Lieb algebra, denoted $TL_n$, is the unital, associative algebra over the ring $\mathbb{Z}[t, t^{-1}]$ generated by
$u_i$, for $1 \leq i \leq n-1$,  and subject to the following relations:

1) $u_i^2 = (-t^2 - t^{-2}) u_i$,   $1 \leq i \leq n-1$;

2) $u_i u_j u_i =  u_i$, for all  $1 \leq i, j \leq n-1$ with $|i - j| = 1$;

3) $u_i u_j = u_j u_i$, for all  $1 \leq i, j \leq n-1$ with $|i - j| > 1$.

In \cite{CCC}, it was proved that  for any $a, b \in \mathbb{Z}[t, t^{-1}]$ the map $\rho_{a, b} \colon SM_n \to TL_n$, which is defined on the generators by,
$$
\rho_{a,b}(\sigma_i) = t^{-1} u_i + t e,~~~\rho_{a,b}(\sigma_i^{-1}) = t u_i + t^{-1} e,~~~\rho_{a,b}(\tau_i) =  a u_i + b e,~~1 \leq i \leq n-1,
$$
where $e$ is the unit element of $TL_n$, is a representation of the singular braid monoid.

\begin{question}
Is it possible to extend $\rho_{a,b}$ to a representation of the group $SB_n$?
\end{question}

\bigskip


\end{document}